\documentclass[a4paper]{article}

\usepackage{amsmath}
\usepackage{amsthm}
\usepackage{amsfonts}
\usepackage[all]{xy}
\SelectTips{cm}{10}
\usepackage{url}
\usepackage{a4wide}
\usepackage{enumitem}
\usepackage{titlesec}

\titleformat{\section}[block]
{\normalfont\Large\filcenter\bfseries}{\thesection.}{.33em}{}
\titlespacing*{\section}
{0pt}{3.5ex plus 1ex minus .2ex}{2.3ex plus .2ex}

\titleformat{\subsection}[runin]
{\normalfont\normalsize\bfseries}{\thesubsection.}{.33em}{}[.]
\titlespacing*{\subsection}
{0pt}{3.25ex plus 1ex minus .2ex}{.5em}

\theoremstyle{plain}
\newtheorem{theorem}{Theorem}
\newtheorem{lemma}[theorem]{Lemma}
\newtheorem{proposition}[theorem]{Proposition}

\theoremstyle{definition}
\newtheorem{definition}[theorem]{Definition}
\newtheorem{example}[theorem]{Example}

\theoremstyle{remark}
\newtheorem*{remark}{Remark}

\newcommand{\bbZ}{\mathbb Z}
\newcommand{\mcC}{\mathcal C}
\newcommand{\mcD}{\mathcal D}
\newcommand{\mcI}{\mathcal I}
\newcommand{\mcX}{\mathcal X}

\newcommand{\simp}{\mathsf{s}}
\newcommand{\Cat}{\mathsf{Cat}}
\newcommand{\Set}{\mathsf{Set}}
\newcommand{\sSet}{\simp\Set}
\newcommand{\Ab}{\mathsf{Ab}}
\newcommand{\sAb}{\simp\Ab}

\newcommand{\Ch}{\mathsf{Ch}}
\newcommand{\Mod}{\mathsf{Mod}}

\newcommand{\colim}{\operatorname{colim}}
\newcommand{\const}[1]{\underline{#1}}

\newcommand{\Id}{\operatorname{Id}}

\newcommand{\lcomma}{\mathord{{/}\hspace{-.3em}{/}}}
\renewcommand{\lim}{\operatorname{lim}}
\newcommand{\rH}{\widetilde H}
\newcommand{\op}{\mathrm{op}}

\newcommand{\la}{\leftarrow}
\newcommand{\Ra}{\Rightarrow}
\newcommand{\lRa}{\Longrightarrow}
\newcommand{\xra}[1]{\xrightarrow{#1}}

\newcommand{\horn}[2]{%
\xy
<0pt,-\the\fontdimen22\textfont2>;p+<.1em,0em>:
{\ar@{-}(0,0);(3,7)},
{\ar@{-}(3,7);(5,0)},
{\ar@{-}(3.2,7);(5.2,0)},
{\ar@{-}(3.4,7);(5.4,0)}
\endxy\;\!{}^{#1}_{#2}}

\title{Pushouts of categories, derived limits and colimits\thanks{%
The research was supported by the grant P201/11/0528 of the Czech Science Foundation (GA \v CR).
\vskip .5ex \noindent
2010 \emph{Mathematics Subject Classification}. Primary 18G10; Secondary 18G15.
\vskip .5ex \noindent
\emph{Key words and phrases}. Derived limit, derived colimit, pushout of categories, Mayer--Vietoris sequence.
}}

\author
{Luk\'a\v{s} Vok\v{r}\'{\i}nek}

\begin{document}

\maketitle

\begin{abstract}
In a paper by Ford, it is claimed that to any pushout square of categories with all involved functors injective, there is associated an exact ``Mayer--Vietoris'' sequence of derived (co)limits. We provide a counter-example to this general statement. Further, we construct the Mayer--Vietoris sequence under some restrictions that cover the well-known case of a pushout square of group monomorphisms.
\end{abstract}

\section{Introduction}

In this paper, we are interested in derived functors of the colimit functor $\colim^\mcC \colon [\mcC,\Ab] \to \Ab$, defined on the category $[\mcC,\Ab]$ of functors $\mcC \to \Ab$; we denote them by $\colim^\mcC_n = L_n \colim^\mcC$.

We study the dependence of these derived colimits on the indexing category $\mcC$. More concretely, assume that $\mcC$ is a pushout of two subcategories $\mcC_1 ,\, \mcC_2 \subseteq \mcC$ along their common subcategory $\mcC_0$. It was claimed in \cite{Ford} that in this situation, there exists, for $M \colon \mcC \to \Ab$, an exact \emph{Mayer--Vietoris sequence}
\[\cdots \to \colim^{\mcC}_{n+1} M \xra\partial \colim^{\mcC_0}_n M \to \colim^{\mcC_1}_n M \oplus \colim^{\mcC_2}_n M \to \colim^{\mcC}_n M \xra\partial \colim^{\mcC_0}_{n-1} M \to \cdots,\]
where the derived colimits are taken of the restriction of $M$ to the various subcategories of $\mcC$ and where the maps are induced by the inclusions among the subcategories, with the exception of those denoted by $\partial$. Another Mayer--Vietoris sequence was claimed for derived limits. The statements do not hold in this generality and we provide a simple counter-example in Section~\ref{s:counter-example}.

After that, we prove the following restricted version of the theorem.

\begin{theorem}\label{t:main}
Let
\begin{equation}\label{eq:Cat}\tag{Cat}
\xymatrix@C=30pt{
\mcC_0 \ar[r]^-{F_1} \ar[d]_-{F_2} & \mcC_1 \ar[d]^-{I_1} \\
\mcC_2 \ar[r]_-{I_2} & \mcC
}\end{equation}
be a pushout square of categories with at least one of $F_1$, $F_2$ injective on objects and all $I_1$, $I_2$ and $I_1F_1=I_2F_2$ \emph{local coverings}. Then an exact Mayer--Vietoris sequence associated with this square exists.
\end{theorem}

Local coverings will be defined later in Section~\ref{s:local_coverings}. They are generalizations of monomorphisms of groups. In particular, the above theorem generalizes the well-known special case of groups, see~\cite[1B.12]{Hatcher}.

\section{Preliminaries on derived limits and colimits}

Let $\mcC$ be a small category. We denote by $\mcC{-}\Mod$ the category of functors $\mcC \to \Ab$, also called \emph{left $\mcC$-modules}. Dually, $\Mod{-}\mcC$ is the category of \emph{right $\mcC$-modules}, i.e.\ functors $\mcC^\op \to \Ab$. It is well-known that $\mcC{-}\Mod$ has enough projectives and injectives and thus, one may define derived functors. In this paper, we are interested in the left derived functors $\colim^\mcC_n=L_n\colim^\mcC$ of the colimit functor $\colim^\mcC \colon \mcC{-}\Mod \to \Ab$ and right derived functors $\lim_{\mcC^\op}^n=R^n\lim_{\mcC^\op}$ of the limit functor $\lim_{\mcC^\op} \colon \Mod{-}\mcC \to \Ab$.

We give a concrete model for derived colimits based on the tensor product of $\mcC$-modules. We ignore the dual situation of derived limits completely, leaving it to the reader. For $L \in \Mod{-}\mcC$ and $M \in \mcC{-}\Mod$, we define the \emph{tensor product} $L \otimes_\mcC M$ to be the coequalizer of the action of $\mcC$ on the two factors of the tensor product,
\[\xymatrix@C=3pc{
	\bigoplus\limits_{c_0,c_1 \in \mcC} Lc_1 \otimes \bbZ\mcC(c_0,c_1) \otimes Mc_0 \ar@<.2pc>[r] \ar@<-.2pc>[r] & \bigoplus\limits_{c\in\mcC} Lc \otimes Mc
}\]
with the two maps sending $x \otimes f \otimes y$ to $xf\otimes y$ and to $x \otimes fy$.

In this way, the colimit functor can be seen as a tensor product $\colim^\mcC M \cong \underline\bbZ \otimes_\mcC M$, where $\underline\bbZ$ denotes the constant diagram with the value $\bbZ$. By standard arguments, the derived colimit is defined by replacing either of $\underline\bbZ$ and $M$ by a projective resolution. There is a canonical choice of a resolution of $\underline\bbZ$, given by $\bbZ N({-}/\mcC)$, the chain complex associated with the diagram of the nerves of the under categories $c/\mcC$. We will thus use as a definition
\[\colim^\mcC_n M = H_n(\bbZ N(-/\mcC) \otimes_\mcC M).\]

Now, we discuss the dependence of the $\colim_n$'s on the indexing category. For each functor $F \colon \mcD \to \mcC$, there is an induced restriction functor $F^* \colon \Mod{-}\mcC \to \Mod{-}\mcD$; it admits both left and right adjoints, the Kan extension functors. Concretely, the left adjoint $F_! \colon \Mod{-}\mcD \to \Mod{-}\mcC$ is given by
\[F_!L(c)=L \otimes_\mcD \bbZ\mcC(c,F{-}).\]

\begin{lemma}\label{l:tensor_Kan}
The following formula holds
\[F_!L \otimes_\mcC M \cong L \otimes_\mcD F^*M.\]
\end{lemma}

\begin{proof}
By the associativity of the tensor product, we have
\begin{align*}
F_!L \otimes_\mcC M & \cong (L \otimes_\mcD \bbZ\mcC({-},F{-})) \otimes_\mcC M \\
& \cong L \otimes_\mcD (\bbZ\mcC({-},F{-}) \otimes_\mcC M) \cong L \otimes_\mcD MF = L \otimes_\mcD F^*M.\qedhere
\end{align*}
\end{proof}

We will now make use of the above lemma to induce maps on derived colimits. The canonical map $\bbZ N(-/\mcD) \to F^*\bbZ N(-/\mcC)$, given by $F$, has an adjunct $F_!\bbZ N(-/\mcD) \to \bbZ N(-/\mcC)$ and upon tensoring with $M$ yields
\[\bbZ N(-/\mcD) \otimes_\mcC F^*M \cong F_!\bbZ N(-/\mcD) \otimes_\mcC M \to \bbZ N(-/\mcC) \otimes_\mcC M\]
which induces in homology the map $F_* \colon \colim^\mcD_n F^*M \to \colim^\mcC_n M$.

Later, we will need the following simple observations.

\begin{lemma}\label{l:derived_colim_under}
There is an isomorphism
\[\colim^\mcD_n F^*M \cong H_n(\bbZ N(-/F) \otimes_\mcC M),\]
where $c/F$ is the under category of $F$. In addition, $\bbZ N(-/F)$ is a chain complex of projectives.
\end{lemma}

\begin{proof}
It is simple to verify that $\bbZ N(-/F) \cong F_! \bbZ N(-/\mcD)$. The first claim thus follows from Lemma~\ref{l:tensor_Kan}. Since $F_!$ is left adjoint to an exact functor $F^*$, it preserves projectives, proving the second claim.
\end{proof}

\section{A counter-example}\label{s:counter-example}

In this section, we present a counter-example to the main theorem of \cite{Ford}. Thus, let \eqref{eq:Cat} be a pushout square in which both functors $F_1$, $F_2$ are injective. In our example, we take $M=\const\bbZ$, the constant functor with value $\bbZ$. In this case, we have $\colim^\mcC_n\const\bbZ \cong H_n(N\mcC;\bbZ)$, the singular homology groups of the nerve of $\mcC$.

Let $\mcC_0$ be an arbitrary connected category for which the reduced homology $\rH_*N\mcC_0$ is non-trivial, e.g.\ $\mcC_0$ could be the group of integers. Then define $\mcC_1$ to be the category obtained from $\mcC_0$ by adjoining a disjoint initial object $0$, while $\mcC_2$ is obtained by adjoining a disjoint terminal object $1$. Then it is easy to see that the pushout $\mcC$ contains both $\mcC_1$ and $\mcC_2$ as full subcategories and there is a unique morphism $0\to 1$, making $0$ an initial object and $1$ a terminal object.

We consider the image of \eqref{eq:Cat} under the nerve functor $N\colon\Cat\to\sSet$. By the above considerations, $N\mcC_1$, $N\mcC_2$ and $N\mcC$ are all contractible, having thus zero reduced (co)homology. The exactness of any kind of Mayer--Vietoris sequence contradicts the triviality of these groups and non-triviality of $\rH_*N\mcC_0$.

\section{Pushouts of categories}

We reformulate the question of the existence of a Mayer--Vietoris sequence in terms of categories and spaces. We consider a square of categories \eqref{eq:Cat} as above and, for any object $c\in\mcC$, form the square of nerves of under categories
\begin{equation}\label{eq:sSet}\tag{sSet-$\mcC$}
\xymatrix@C=30pt{
N(c/I_0) \ar[r]^-{(F_1)_*} \ar[d]_-{(F_2)_*} & N(c/I_1) \ar[d]^-{(I_1)_*} \\
N(c/I_2) \ar[r]_-{(I_2)_*} & N(c/\mcC)
}\end{equation}
where $I_0\colon\mcC_0\to\mcC$ is the composite functor in the original diagram, $I_0=I_1F_1=I_2F_2$. The significance of this diagram for the original problem lies in the following proposition, where we remind that a square is called \emph{homotopy cocartesian} or a \emph{homotopy pushout square} if the induced map from the homotopy pushout to the bottom right corner is a weak equivalence.

\begin{proposition}\label{p:Mayer_Vietoris}
If the square \eqref{eq:sSet} is homotopy cocartesian then \eqref{eq:Cat} induces a Mayer--Vietoris sequence of derived (co)limits.
\end{proposition}

\begin{proof}
We will work out the case of derived colimits, the case of limits is dual. Consider the category $\Ch$ of non-negatively graded chain complexes with the projective model structure. The chain complex functor $\bbZ \colon \sSet \to \Ch$ is known to preserve homotopy colimits (the corresponding functor $\bbZ \colon \sSet \to \sAb$ is left Quillen and the Moore complex functor $\sAb \to \Ch$ is naturally quasi-isomorphic to the normalized chain complex functor; the latter is a Quillen equivalence) and therefore, the induced square of chain complexes
\begin{equation}\label{eq:Ch}\tag{Ch-$\mcC$}
\xymatrix@C=30pt{
\bbZ N(c/I_0) \ar[r]^-{(F_1)_*} \ar[d]_-{(F_2)_*} & \bbZ N(c/I_1) \ar[d]^-{(I_1)_*} \\
\bbZ N(c/I_2) \ar[r]_-{(I_2)_*} & \bbZ N(c/\mcC)
}\end{equation}
is also homotopy cocartesian. By Lemma~\ref{l:derived_colim_under}, all $\mcC$-modules appearing in \eqref{eq:Ch} are projective. Tensoring with $M$ over $\mcC$ thus yields another homotopy pushout square of chain complexes.
The homology groups of $\bbZ N(-/\mcC) \otimes_\mcC M$ are the derived colimits $\colim^\mcC_*M$ by definition. For each $k=0,1,2$, the homology groups of the chain complex $\bbZ N(-/I_k) \otimes_\mcC M$ are the derived colimits $\colim^{\mcC_k}_*(I_k^*M)$ by Lemma~\ref{l:derived_colim_under}. Thus, the Mayer--Vietoris sequence associated with the homotopy pushout square $\eqref{eq:Ch} \otimes_\mcC M$ is the required sequence.
\end{proof}

We will need a small lemma regarding the passage from \eqref{eq:Cat} to the induced square of under categories:
\begin{equation}\label{eq:CatC}\tag{Cat-$\mcC$}
\xymatrix@C=30pt{
c/I_0 \ar[r]^-{(F_1)_*} \ar[d]_-{(F_2)_*} & c/I_1 \ar[d]^-{(I_1)_*} \\
c/I_2 \ar[r]_-{(I_2)_*} & c/\mcC
}\end{equation}

\begin{lemma}
If \eqref{eq:Cat} is a pushout square, then so is \eqref{eq:CatC}.
\end{lemma}

\begin{proof}
The functor $\Cat/\mcC\to\Cat$, $(\mcD \xra{F} \mcC) \mapsto c/F$, preserves colimits since it is a left adjoint: Let $Y_c\colon\mcC\to\Set\subseteq\Cat$ denote the representable functor $\mcC(c,-)$, thought of as a functor to $\Cat$ by viewing each set as a discrete category. The value of the right adjoint $\Cat \to \Cat/\mcC$ on a category $\mcX$ is given by the ``lax'' under category $Y_c \lcomma \mcX$, whose objects are pairs $(a,\varphi)$ with $a \in \mcC$ and $\varphi \colon Y_c(a) \to \mcX$ and whose morphisms are pairs $(f,\tau)$ with $f \colon a \to b$ and $\tau \colon \varphi \Ra \psi f_*$ a transformation. It is thought of as an object of $\Cat/\mcC$ via $Y_c \lcomma \mcX \to \mcC$, $(a,\varphi) \mapsto a$.
\end{proof}

\begin{remark}
An analogous lemma holds for the Grothendieck construction in $\Cat$: the under category of a Grothendieck construction is a Grothendieck construction of the under categories. Since the nerve functor takes, up to weak homotopy equivalence, Grothendieck constructions to homotopy colimits by \cite{Thomason:hocolim}, one obtains a Mayer--Vietoris sequence for derived colimits over a Grothendieck construction of a diagram $\mcC_1 \la \mcC_0 \to \mcC_2$.

In fact, it is possible to extend the above to more general shapes than just spans of categories, at the cost of replacing the Mayer--Vietoris sequence by a spectral sequence for homology groups of a homotopy colimit of chain complexes. Denoting the Grothendieck construction by $\mcC=\smallint^{i \in \mcI}\mcC_i$, this becomes
\[\colim^\mcI_p (i \mapsto \colim^{\mcC_i}_q M) \lRa \colim^\mcC_{p+q}M.\]
Such a spectral sequence was obtained in a more general context in \cite{Pirashvili}.

However, the main point of our theorem is to simplify the computation of derived colimits over a \emph{given} category $\mcC$ by splitting it into subcategories, a procedure rather inverse to the above.
\end{remark}

\section{Local coverings of categories}\label{s:local_coverings}

There is an obvious candidate for a condition on the square \eqref{eq:Cat} that ensures that \eqref{eq:sSet} is homotopy cocartesian. Namely, by a theorem of Thomason \cite{Thomason:cat_sset} on the Quillen equivalence $\Cat \simeq_Q \sSet$, we may translate this question to $\Cat$ with Thomason's model structure. In a left proper model category such as $\Cat$, a pushout square is homotopy cocartesian if one of the functors $F_1$, $F_2$ is a cofibration. This, however, is too restrictive, since cofibrations in $\Cat$ are rather rare. In particular, this does not apply to the case of a pushout square of group monomorphisms that we would like to generalize.

The main property of group monomorphisms used in the proof of \cite[1B.12]{Hatcher} is that on nerves, they give coverings. We will now generalize this notion from groups to arbitrary categories.

\begin{definition}
A functor $F\colon\mcD\to\mcC$ is said to be a \emph{local covering} if each under category $c/F$ is homotopically discrete, i.e.\ weakly homotopy equivalent to a discrete space.
\end{definition}

\begin{example}
Let $F\colon\mcD\to\mcC$ be a faithful functor from a groupoid $\mcD$ into a category $\mcC$ in which every morphism is epi. Then $F$ is a local covering, as follows from an easily verified fact that each $c/F$ is a groupoid in which every automorphism is the identity.

In particular, any faithful functor between groupoids is a local covering.
\end{example}

\begin{remark}
The condition from the definition is named a local covering since it does not give any (global) information on $F_* \colon N\mcD \to N\mcC$. We say that $F$ is a \emph{covering} if, in addition to the local covering condition, for every arrow $c'\to c$, the induced functor $c/F\to c'/F$ is bijective on $\pi_0$, i.e.\ $N(c/F)\to N(c'/F)$ is a weak homotopy equivalence. Quillen's Theorem~B then shows that the homotopy fibre of $F_* \colon N\mcD \to N\mcC$ is also homotopically discrete and thus, $F_*$ is a covering.

If $F_k$ in \eqref{eq:Cat} is a covering then Quillen's Theorem~B also yields a fibration sequence
\[N(c_k/F_k)\to N(c/I_0)\to N(c/I_k);\]
when $I_k$ is a (local) covering, both the base and the fibre are homotopy discrete, hence also the total space, and $I_0$ is a (local) covering too. We believe that if both $F_1$, $F_2$ are coverings then so are $I_1$, $I_2$ and, by the above, also $I_0$.
\end{remark}

We remind that a \emph{homotopy $1$-type} is a Kan complex $K$ with $\pi_nK=0$ for all $n>1$. It is possible to associate to every simplicial set $X$ a homotopy $1$-type $P_1X$ (the first Postnikov section of $X$) by killing all higher homotopy groups of $X$. For another point of view on this construction, see Proposition~\ref{p:one_type}. The following lemma contains the heart of the proof of Theorem~\ref{t:main}.

\begin{lemma}
For any pushout square of categories \eqref{eq:Cat} with at least one of $F_1$, $F_2$ injective on objects, the map from the homotopy pushout in \eqref{eq:sSet} to its bottom right corner $N(c/\mcC)$ is a $2$-equivalence, i.e.\ the associated map of homotopy $1$-types is a homotopy equivalence.
\end{lemma}

\begin{proof}
By Proposition~\ref{p:one_type}, there is a Quillen equivalence
\[\xymatrix{
\tau_1 \colon \sSet_\mathrm{loc} \ar@<2pt>[r] & \Cat_\mathrm{loc} \:\!:\! N. \ar@<2pt>[l]
}\]
The condition from the statement is equivalent to the square \eqref{eq:sSet} being homotopy cocartesian in $\sSet_\mathrm{loc}$. This happens if and only if its image under $\tau_1$ is homotopy cocartesian. In the localized model structure, $\tau_1P_1N\mcD \simeq \tau_1N\mcD \cong \mcD$. Thus, we are left to verify that \eqref{eq:CatC} is homotopy cocartesian in $\Cat_\mathrm{loc}$ and for that matter, we localize this square. This produces a pushout square since the localization $\mcD \mapsto \mcD[\mcD^{-1}]$ commutes with colimits (it is a reflection of categories into groupoids). It is known that a pushout square is homotopy cocartesian if one of the maps $(F_1)_*$, $(F_2)_*$ is a cofibration. In $\Cat_\mathrm{loc}$, these are exactly functors that are injective on objects and it is trivial to check that this condition for $F_k$ implies that for $(F_k)_* \colon c/I_0 \to c/I_k$.
\end{proof}

\subsection*{Proof of Theorem~\ref{t:main}}
By definition of local coverings, all spaces appearing in the square \eqref{eq:sSet} are homotopy discrete. Therefore, up to homotopy, its homotopy pushout $P$ is $1$-dimensional and thus a homotopy $1$-type. The contractible $N(c/\mcC)$ is also a homotopy $1$-type and the previous lemma gives $P\simeq N(c/\mcC)$. Hence, $\eqref{eq:sSet}$ is homotopy cocartesian and by Proposition~\ref{p:Mayer_Vietoris}, it induces a Mayer--Vietoris sequence.\qed

\appendix
\section{Homotopy $1$-types and groupoids}\label{s:one_type}

We were not able to find the following statements anywhere explicitly.

\begin{proposition}\label{p:one_type}
There exist left Bousfield localizations of the categorical model structure on $\Cat$ and of the Quillen model structure on $\sSet$ such that
\begin{itemize}[topsep=2pt,itemsep=2pt,parsep=2pt,leftmargin=\parindent,label=$\bullet$]
\item
	local objects in $\Cat$ are exactly the groupoids and the localization is $\mcD \to \mcD[\mcD^{-1}]$,
\item
	local objects in $\sSet$ are exactly homotopy $1$-types and the localization is $X \to P_1X$,
\item
	there is a Quillen equivalence $\xymatrix@1{\tau_1 \colon \sSet_\mathrm{loc} \ar@<2pt>[r] & \Cat_\mathrm{loc} \:\!:\! N \ar@<2pt>[l]}$.
\end{itemize}
\end{proposition}

\begin{proof}
Using \cite{Hirschhorn}, the categorical model structure on $\Cat$ is left Bousfield localized with respect to $\tau_1\horn20 \to \tau_1\Delta^2$ and $\tau_1\horn22 \to \tau_1\Delta^2$, the images under $\tau_1$ of the two outer horn inclusions. Similarly, the Quillen model structure is localized with respect to $\partial\Delta^n \to \Delta^n$ for all $n>2$.

The first two points are easy to verify. For the last point, we observe that the localization of $\tau_1X$ is the fundamental groupoid $\Pi_1X$, i.e.\ $\tau_1 \simeq \Pi_1$, and thus, $\tau_1$ preserves and reflects local equivalences. Consequently, it is enough to verify that the derived counit is a weak equivalence. The counit $\tau_1N \to \Id$ is known to be an isomorphism and on local objects, it is weakly equivalent to the derived counit since all simplicial sets are cofibrant.
\end{proof}

\vskip 20pt
\vfill
\vbox{\footnotesize%
\noindent\begin{minipage}[t]{0.45\textwidth}
{\scshape Luk\'a\v{s} Vok\v{r}\'inek}\\
Department of Mathematics and Statistics,\\
Masaryk University,\\
Kotl\'a\v{r}sk\'a~2, 611~37~Brno,\\
Czech Republic
\end{minipage}}

\end{document}